\documentclass[pdflatex,sn-mathphys-num]{sn-jnl}

\usepackage{graphicx}%
\usepackage{multirow}%
\usepackage{amsmath,amssymb,amsfonts,mathrsfs}%
\usepackage{amsthm}%
\usepackage{amssymb}%
\usepackage{mathrsfs}%
\usepackage[title]{appendix}%
\usepackage{xcolor}%
\usepackage{textcomp}%
\usepackage{manyfoot}%
\usepackage{booktabs}%
\usepackage{algorithm}%
\usepackage{algorithmicx}%
\usepackage{algpseudocode}%
\usepackage{listings}%

\usepackage{tikz-cd}
\usepackage{mathrsfs}
\usepackage[all,cmtip]{xy}
\usepackage{amsmath}
\usepackage{amssymb}

\newcommand{\sDn}{\mathscr{D} _n}
\newcommand{\Cnz}{\mathbb{C}^n _z}
\newcommand{\D}{\mathbb{D}}
\newcommand{\C}{\mathbb{C}}
\newcommand{\uw}{\underline{w}}
\newcommand{\Cnzs}{\mathbb{C}^n _{z,s}}

\theoremstyle{thmstyleone}%
\newtheorem{theorem}{Theorem}
\newtheorem{lemma}{Lemma}%
\newtheorem{corollary}{Corollary}%
\theoremstyle{thmstyletwo}%

\theoremstyle{thmstylethree}%
\newtheorem{definition}{Definition}%

\begin{document}

\title[Article Title]{On the automorphisms of Pick bodies}


\author*[1,2]{\fnm{Anindya} \sur{Biswas}}\email{anindyab132652@gmail.com}

\affil*[1]{\orgdiv{Department of Mathematics and Statistics}, \orgname{Masaryk University}, \orgaddress{\city{Brno}, \postcode{61137}, \country{Czechia}}}


\abstract{We determine the biholomorphic automorphisms of the Pick body associated with finitely many distinct points in a Carath\'eodory hyperbolic domain. Every automorphism is the composition of a common disc automorphism applied to all coordinates and a coordinate permutation preserving the Pick body. The proof uses the quotient Banach algebra of interpolation data, its generalized von Neumann inequality, and the fact that its symmetric part consists precisely of the constant vectors. For nodes in a polydisc, we introduce minimal effective coordinate tuples and identify the admissible permutations with the automorphisms of a lower-dimensional polydisc that preserve the projected node set. This yields a direct-product description of the automorphism group, with one factor equal to the automorphism group of the disc and the other finite. We also give the resulting descriptions for nodes in the disc and the bidisc.}




\maketitle

\section{Introduction}\label{sec:introduction}

Finite interpolation by bounded holomorphic functions associates a geometric object with each finite set of nodes: the collection of their possible value vectors subject to a uniform norm bound. Its interior is the Pick body. This construction records the interpolation problem as the unit ball of a finite-dimensional quotient Banach algebra. The geometry of that ball therefore reflects both the analytic restrictions on interpolation and the algebraic structure of the space of data. A natural question is whether its biholomorphic symmetries can be described in terms of these two structures.

The work of Cole, Lewis and Wermer places interpolation bodies in the general framework of uniform algebras. Their study of Pick conditions and von Neumann inequalities \cite{CLW1992} links finite interpolation to quotient norms and to the geometry of hyperconvex sets. In \cite{CLW1993}, they characterize the Pick bodies associated with finite subsets of the disc by combining hyperconvexity with a boundary condition. Cole and Wermer \cite{CW1993} further develop the connection with representations of uniform algebras on Hilbert space. Their subsequent papers investigate interpolation bodies for finite node sets in the bidisc \cite{CW1996} and establish algebraic and real-analytic properties of their boundaries \cite{CW1997}. More recently, the current author \cite{Biswas2025} associates invariant functions with Pick bodies that encode interpolation solvability and generalize the Carath\'eodory pseudodistance.

The generalized von Neumann inequality is particularly useful in this passage from interpolation to geometry. In the present setting it states that a polynomial evaluated at any finite collection of elements of the closed unit ball of the interpolation algebra has norm at most the supremum of that polynomial on the corresponding closed polydisc. The inequality follows by lifting interpolation data to bounded holomorphic functions and then applying the scalar supremum estimate. For polynomials bounded by one on the polydisc, it says that the closed interpolation body is stable under polynomial evaluation: the property called hyperconvexity in this setting. Thus it expresses an analytic constraint on interpolation in the intrinsic language of the quotient algebra. The relationship between these inequalities and Pick conditions is central to the development described in \cite{CLW1992,CLW1993,CW1993}. Here its one-variable consequence allows us to use Arazy's structure theory for isometries \cite{Araz94}. This is the step that turns a linear symmetry of the Pick body into an algebraic symmetry.

The other ingredient is the symmetric part of a Banach space, defined through the complete holomorphic vector fields on its open unit ball. Arazy and Solel \cite{ArazSol90} explain how this structure relates holomorphic automorphisms to isometries and Jordan algebra operations. We show that the symmetric part of an interpolation algebra over a connected domain consists exactly of the constant vectors. Consequently, the orbit of the origin is the diagonal disc. Composing an arbitrary automorphism with a diagonal M\"obius transformation then produces an automorphism fixing the origin, to which Cartan's linearity theorem applies (see \cite{GarciaPalacios}). The Banach algebra background for this approach is developed in \cite{BonDun73}.

Our first result shows that every automorphism of a Pick body is obtained by applying the same automorphism of the disc to every coordinate and then permuting the coordinates by a permutation preserving the body. Both factors are uniquely determined. In particular, the automorphism group is the direct product of the disc automorphism group and a finite group of coordinate permutations. The conclusion applies to any Carath\'eodory hyperbolic domain and does not require an explicit criterion for solving its finite interpolation problems.

For nodes in a polydisc, the finite factor admits a geometric description. We call a coordinate tuple effective if the nodes lie on the graph of a holomorphic map over those coordinates, and we choose such a tuple of smallest possible length. The Pick body is unchanged when one passes to the projected nodes. The fixed-point theory of holomorphic maps on bounded convex domains \cite{Abate89}, together with the description of holomorphic retracts of the polydisc due to Heath and Suffridge \cite{HeathSuff81}, ensures that every holomorphic selfmap of the smaller polydisc fixing all projected nodes is the identity. Norm-preserving extensions of bounded holomorphic functions and their relation to retracts in the tridisc are studied by Kosi\'nski and McCarthy \cite{KosMcCar20}. We use this rigidity property to identify the admissible coordinate permutations with the automorphisms of the smaller polydisc preserving the projected node set.

Section~\ref{sec:results} contains the mathematical development. We first establish the general automorphism theorem. We then prove the polydisc description using minimal effective tuples and conclude with the disc and bidisc cases.

\section{Automorphisms of Pick bodies}\label{sec:results}

Suppose that $\Omega$ is a Carath\'eodory hyperbolic domain, that is, the Banach algebra $H^\infty (\Omega)$ of all bounded holomorphic functions separates the points of $\Omega$. We fix mutually distinct points $z_1,\ldots, z_n$ in $\Omega$ ($n\geq 2$) and we consider the Pick body 

\begin{align}\label{PickBody}
	\mathscr{D}_{\Omega}(\underline{z})&=\mathscr{D}_n\\
	&=\{(f(z_1),\ldots,f(z_n)):f\in H^\infty (\Omega), ||f||<1\}\notag
\end{align}
where $\underline{z}$ stands for the tuple $(z_1,\ldots,z_n)$. It is a convex balanced domain the Euclidean space $\mathbb{C}^n$. Let $\mu_{\Omega,\underline{z}}=\mu_n$ denote its Minkowski function. It is known that $(\mathbb{C}^n,\mu_n)$ is a finite dimensional Banach algebra with unit $\mathbf{1}=(1,\ldots,1)$ and it will be denoted by $\mathbb{C}^n _{\underline{z}}$. The domain (\ref{PickBody}) is the open unit ball of $\Cnz$. We will study the automorphism group $Aut(\sDn)$, that is, the collection of biholomorphic maps on $\sDn$. We denote its Lie algebra by $aut(\sDn)$. It is the collection of all completely integrable holomorphic vector fields on $\sDn$, that is, an $h\in \mathcal{O}(\sDn,\Cnz)$ (all holomorphic maps from $\sDn$ to $\Cnz$) is in $aut(\sDn)$ if and only if there is a one-parameter subgroup $\{\varphi_t\}_{t\in \mathbb{R}}$ of $Aut(\sDn)$ such that
\begin{align*}
	\frac{\partial}{\partial t}\varphi_t (\uw)=h(\varphi_t (z)),\,\,\text{and}\,\,\varphi(0,\uw)=\varphi_0 (\uw)=\uw\,\,\text{for all}\,\,\uw\in \sDn.
\end{align*}

By \cite{Biswas2025}, for each $a$ in the open unit disc $\D$ and $\theta\in \mathbb{R}$, the automorphism $\varphi(\lambda)=e^{i\theta}\frac{a+\lambda}{1+\overline{a}\lambda}$ of $\D$ gives an element $\varPhi_\varphi$ of $Aut(\sDn)$ which is defined by $\varPhi_\varphi(w_1,\ldots,w_n)=(\varphi(w_1),\ldots,\varphi(w_n))$.

Consider the map $\varphi:\mathbb{R}\times\sDn\rightarrow\sDn$ given by 
\begin{align}\label{IntCurv}
	\varphi(t,\underline{w})&=\varphi_t (\underline{w})\\
	&=(tanh(t) +\uw)(1+tanh(t )\uw)^{-1}.
\end{align}
The map (\ref{IntCurv}) is the integral curve of the vector field $h_{\mathbf{1}}(\uw)=\mathbf{1}-\uw ^2$ (See Proposition 2.2 in \cite{Araz94}, also Section 50 in \cite{BonDun73}). Since $\Cnz$ is isometrically isomorphic to $H^\infty (\Omega)/I(\underline{z})$ ($I(\underline{z})$ is the collection of all functions in $H^\infty (\Omega)$ vanishing at $z_1,\ldots, z_n$), it satisfies a multi-variable generalized version of the von Neumann inequality (See Theorem 4.5 in \cite{Araz94}). From this and Proposition 2.2. $(ii)$ in \cite{Araz94}, we conclude that $h_\mathbf{1}\in aut(\sDn)$. Now we look at the symmetric part $\Cnzs$ of $\Cnz$
\begin{align*}
	\Cnzs=aut(\sDn)(\mathbf{0})=\{h(\mathbf{0}):h\in aut(\sDn)\}.
\end{align*}
See \cite{ArazSol90} for more details. In particular, we get $h_\mathbf{1}(\mathbf{0})=\mathbf{1}\in \Cnzs$.

An element $\uw\in \Cnz$ is called hermitian if $\mu_{n}(exp(it\uw))=1$ for all $t\in \mathbb{R}$, that is $exp(it\uw)\in \partial\sDn$ for all $t\in \mathbb{R}$. Let us denote the real vector space of all hermitian elements of $\Cnz$ by $H_{\underline{z}}$. Since $\mathbf{1} \in  \Cnzs$, Proposition 2.4 in \cite{Araz94} gives $H_{\underline{z}}+i H_{\underline{z}}=\Cnzs$. We now find $H_{\underline{z}}$.

\begin{lemma}\label{SymmetricPart}
	$H_{\underline{z}}=\mathbb{R}\cdot \mathbf{1}$ and $\Cnzs=\C\cdot \mathbf{1}$.
\end{lemma}
\begin{proof}
	It is clear from the definition of hermitian elements and of $\partial\sDn$ that $\mathbb{R}\cdot \mathbf{1}\subset H_{\underline{z}}$.  Now suppose $\uw=(w_1,\ldots,w_n)\in \Cnz$ and $exp(it\uw)=(exp(itw_1),\ldots, exp(itw_n))\in \partial\sDn$ for all $t\in \mathbb{R}$. Since, $\partial\sDn$ is contained in the closed polydisc $\overline{\D}^n$, we conclude that each $w_j$ is real and hence each $exp(itw_j)$ is unimodular. Since such data can be interpolated only by constant functions we conclude $exp(itw_1)=\ldots=exp(itw_n)$ for all $t\in \mathbb{R}$. This is sufficient to conclude the proof of the lemma.
\end{proof}

The symmteric part of $\sDn$ (see \cite{ArazSol90}) is defined to be $\mathscr{D}_{n,s}=\sDn\cap \Cnzs$. Thus $\mathscr{D}_{n,s}=\{\lambda\cdot\mathbf{1}:|\lambda|<1\}$.

By Theorem 7 at page $272$ in \cite{BonDun73} (more precisely, by the proof of it), there is a Hilbert space $\mathcal{H}$ such that $\Cnz$ is a closed unital subalgebra of $B(\mathcal{H})$ (space of all bounded linear operators on $\mathcal{H}$). Therefore we may use Corollary 2.9 and Theorem 2.2.$(ii)$ in \cite{ArazSol90} to conclude that $Aut(\sDn)(\mathbf{0})=\{\varPhi(\mathbf{0}):\varPhi\in Aut(\sDn) \}=\{\lambda\cdot\mathbf{1}:|\lambda|<1\}$. With this in our hand, we can now give a description of $Aut(\sDn)$.

\begin{theorem}\label{AutOnCara}
	For a given $\varPhi\in Aut(\sDn)$, there is a $\varphi\in Aut(\D)$ and a permutation $\sigma$ of $\{1,\ldots, n\}$ such that $\varPhi(w_1,\ldots,w_n)=(\varphi(w_{\sigma(1)}),\ldots, \varphi(w_{\sigma(n)}))$ for all $(w_1,\ldots, w_n)\in \sDn$.
\end{theorem}
\begin{proof}
	By Lemma \ref{SymmetricPart}, $\varPhi(\mathbf{0})=(a,\ldots,a)$ for some $a\in \D$. Let $F=\varPhi_a \circ \varPhi$ where $\varPhi_a$ sends $(w_1,\ldots,w_n)$ to $(\frac{a-w_1}{1-\overline{a}w_1},\ldots, \frac{a-w_n}{1-\overline{a}w_n})$. So $F(\mathbf{0})=\mathbf{0}$ and by Proposition 5.6.2 in \cite{GarciaPalacios}, the automorphism $F$ is linear on $\Cnz$. Linearity of $F$ together with the fact that $F(\partial\sDn)=\partial\sDn$ gives that $F$ is a surjective isometry on $\Cnz$. We now use Theorem 3.1 in \cite{Araz94} to conclude
	\begin{enumerate}
		\item $F(\mathbf{1})=e^{i\theta}\mathbf{1}$ for some $\theta\in \mathbb{R}$.
		\item $F(w_1,\ldots,w_n)=e^{i\theta}\psi(w_1,\ldots,w_n)$ for some isometric Jordan isomorphism.
	\end{enumerate}
	Our Jordan isomorphism satisfies the following.
	\begin{enumerate}
		\item Since $\Cnz$ is a commutative Banach algebra, $\psi(\underline{a}\cdot\underline{b})=\psi(\underline{a})\cdot\psi(\underline{b})$. Here the product $"\cdot"$ is componentwise.
		\item $\psi(\mathbf{1})=\mathbf{1}$.
		\end{enumerate}
		Now we consider the idempotent $\underline{e}_j =(e_{j1},\ldots,e_{jn})$ where $e_{jk}=\delta_{jk}$ (the Kronecker delta) and $1\leq j\leq n$. Let $\psi(\underline{e}_j)=\underline{e}^\prime _j$. Then $\underline{e}^\prime _j$ is again an idempotent and we have 
		
		\begin{align}\label{IdempotentUnit}
			\mathbf{1}=\psi(\underline{e} _{1}+\ldots +\underline{e} _{n})=\underline{e}^\prime _{1}+\ldots +\underline{e}^\prime _{n}.
		\end{align}
		Componentwise product implies that the nonzero entries of $\underline{e}^\prime _j$ are $1$. If we now consider an $n\times n$ matrix whose columns are $\underline{e}^\prime _j$, then we see that the matrix is invertible and by (\ref{IdempotentUnit}) the sum of the entries in each row gives $1$. From this, it is easy to conclude that $\{\underline{e}^\prime _j:1\leq j\leq n\}$ and $\{\underline{e} _j:1\leq j\leq n\}$ are the same sets, that is, there is a permutation $\sigma$ of $\{1,\ldots,n\}$ such that $\psi(\underline{e}_j)=\underline{e}_{\sigma^{-1}(j)}$ for all $j$. Linearity of $\psi$ implies $\psi(w_1,\ldots, w_n)=(w_{\sigma(1)},\ldots, w_{\sigma(n)})$ for all $(w_1,\ldots, w_n)$. This is sufficient to conclude the proof.
\end{proof}

Next we consider the case $\Omega=\D^m$ for $m\geq 1$. For any mutually distinct $z_1,\ldots, z_n\in \D^m$ ($n\geq 2$), there is a map $f\in \mathcal{O}(\D^m,\D^m)$ such that $z_1,\ldots, z_n\in Fix(f)=\{w\in \D^m:f(w)=w\}$. For instance, we can always take $f=id_{\D^m}$. It is known that $Fix(f)$ is a holomorphic retract and it is a connected closed submanifold of $\D^m$ (see Corollary 2.5.10 in \cite{Abate89}). 

For $k\in \{1,\ldots,m\}$, we consider a $k$-tuple $I=(i_1,\ldots,i_k)$ where $1\leq i_j\leq m$ and $i_j\neq i_l$ for $j\neq l$. For $w=(w_1,\ldots,w_m)\in \D^m$ we will write $w_I=(w_{i_1},\ldots,w_{i_k})$. We will also write $z_j=(z_{j1},\ldots,z_{jm})$ for $j=1,\ldots,n$ and for any $j$ we will write $z_{j,I}=(z_{ji_1},\ldots,z_{ji_k})$.

Since we are dealing with polydiscs, we know that for each $f\in \mathcal{O}(\D^m,\D^m)$ with $Fix(f)$ containing at least two points, there is a $k\in \{1,\ldots,m\}$, a $k$-tuple $I$ as above and a map $h_f\in \mathcal{O}(\D^k,\D^{m-k})$ such that, upto a permutation of coordinates $Fix(f)=\{(w_I,h_{f}(w_I)):w_I\in \D^k\}$ (see \cite{HeathSuff81}, \cite{KosMcCar20}).

\begin{definition}
	Given mutually distinct $z_1,\ldots, z_n\in \D^m$, we call a $k$-tuple ($1\leq k\leq m$) $I$ an effective tuple for $\{z_1,\ldots,z_n\}$ if there is a map $p\in \mathcal{O}(\D^k,\D^{m-k})$ such that 
	\begin{align}\label{EffectiveTuple}
		z_1,\ldots,z_n\in \{(w_I,p(w_I)):w_I\in \D^k\}.
	\end{align}
	We call $I$ a minimal effective tuple for $\{z_1,\ldots,z_n\}$ if (\ref{EffectiveTuple}) does not hold for any tuple $J$ with the length $|J|$ strictly less than $|I|$.
\end{definition}
Note that for a fixed collection $\{z_1,\ldots,z_n\}$, the minimum of the lengths of the tuples $I$ for which (\ref{EffectiveTuple}) is satisfied is uniquely determined.

Fix a minimal effective tuple $I$ for $\{z_1,\ldots,z_n\}$ and define
\begin{align}
	\mathscr{A}_{\underline{z},I}=\{h\in Aut(\D^{|I|}):h\{z_{j,I}:1\leq j\leq n\}=\{z_{j,I}:1\leq j\leq n\}\}.
\end{align}

For any effective tuple $I$ for $\{z_1,\ldots,z_n\}$ of length $k$, it is easy to see that $\mathscr{D}_{\D^m}(z_1,\ldots,z_n)=\mathscr{D}_{\D^k}(z_{1,I},\ldots,z_{n,I})$. First we describe $\mathscr{A}_{\underline{z},I}$ in terms of Jordan isomorphisms.

\begin{lemma}\label{JordanIsoAndA_zI}
	Let $I$ be a minimal effective tuple for $\{z_1,\ldots,z_n\}\subset \D^m$ and $|I|=k$. Then there is a one-to-one correspondence between the isometric Jordan  isomorphisms of $\Cnz$ and $\mathscr{A}_{\underline{z},I}$.
\end{lemma}
\begin{proof}
	Without loss of generality, we may assume that $I=(1,\ldots,k)$ and write $z_{j,I}=z^\prime _j=(z_{j1},\ldots,z_{j,k})$, $j=1,\ldots, n$. Let $\psi\in Aut(\mathscr{D}_{\D^m}(z_1,\ldots,z_n))$ be an isometric Jordan isomorphism. Then $\psi(\overline{\mathscr{D}}_{\D^m}(z_1,\ldots,z_n))=\overline{\mathscr{D}}_{\D^m}(z_1,\ldots,z_n)$ and there is a permutation $\sigma$ of $\{1,\ldots,n\}$ such that $\psi(w_1,\ldots,w_n)=(w_{\sigma(1)},\ldots,w_{\sigma(n)})$ for all $(w_1,\ldots,w_n)\in\overline{\mathscr{D}}_{\D^m}(z_1,\ldots,z_n)$. Since the coordinate functions have norm one, we have $(z_{1l},\ldots,z_{nl})\in \overline{\mathscr{D}}_{\D^m}(z_1,\ldots,z_n)=\overline{\mathscr{D}}_{\D^k}(z^\prime _1,\ldots,z^\prime _n)$ for all $l=1,\ldots,k$. Consequently, $(z_{\sigma(1)l},\ldots,z_{\sigma(n)l})\in\overline{\mathscr{D}}_{\D^k}(z^\prime _1,\ldots,z^\prime _n)$ for all $l=1,\ldots,k$. For each $l$, we find an $h_l\in\mathcal{O}(\D^k,\D) $ that interpolates the data $z^\prime _j\mapsto z_{\sigma(j)l}$ and construct the map $h=(h_1,\ldots,h_k)\in \mathcal{O}(\D^k,\D^k)$. Note that $h$ maps $z^\prime _j$ to $z^\prime _\sigma(j)$ for $j=1,\ldots,n$. Using $\psi^{-1}$, we can similarly a map $g\in \mathcal{O}(\D^k,\D^k)$ that sends $z^\prime _j$ to $z^\prime _{\sigma^{-1}(j)}$. Hence $z^\prime _1,\ldots,z^\prime _n$ are in $Fix(h\circ g)\cap Fix(g\circ h)$. We now use minimality of $k$ to conclude that $Fix(h\circ g)=Fix(g\circ h)=\D^k$, that is, $g\circ h=h\circ g=id_{\D^k}$. Hence $h\in Aut(\D^k)$ and $h$ maps the set $\{z^\prime _1,\ldots,z^\prime _n\}$ onto itself. Thus, $h\in \mathscr{A}_{\underline{z},I}$. Now suppose the choice of $h_l$ is not unique. Then we may get another map $\tilde{h}$ that maps $z^\prime _j$ to $z^\prime _\sigma(j)$ for $j=1,\ldots,n$ and is in $Aut(\D^k)$. But then we have that $z^\prime _1,\ldots,z^\prime _n$ are in $Fix(h^{-1}\circ \tilde{h})$. Again we invoke minimality of $k$ to conclude that $h=\tilde{h}$. Thus the assignment $\psi\mapsto h$ is injective. To show the surjectivity of this assignment, we pick an $h\in \mathscr{A}_{\underline{z},I}$. This gives a permutation $\sigma$ such that $h(z^\prime _j)=z^\prime _{\sigma(j)}$. Now, $\mathscr{D}_{\D^k}(z^\prime _1,\ldots,z^\prime _n)$ is invariant under the action of $Aut(\D^k)$ and hence $\mathscr{D}_{\D^k}(z^\prime _1,\ldots,z^\prime _n)=\mathscr{D}_{\D^k}(z^\prime _{\sigma(1)},\ldots,z^\prime _{\sigma(n)})$. This implies that $(w_1,\ldots,w_n)\in \mathscr{D}_{\D^k}(z^\prime _1,\ldots,z^\prime _n)$ if and only if $(w_{\sigma(1)},\ldots,w_{\sigma(n)})\in\mathscr{D}_{\D^k}(z^\prime _1,\ldots,z^\prime _n)$. We define 
	
	\begin{align*}
		\psi:\Cnz\rightarrow\Cnz
	\end{align*}
	by $\psi(w_1,\ldots,w_n)=(w_{\sigma(1)},\ldots,w_{\sigma(n)})$. Clearly this is an isometric Jordan isomorphism of $\Cnz$ and $\psi(z_{1l},\ldots,z_{nl})=(h_l (z^\prime _1),\ldots,h_l (z^\prime _n))$ for $l=1,\ldots,k$. 
\end{proof}
\begin{theorem}\label{AutForPolydisc}
	There is a one-to-one correspondence between $Aut(\mathscr{D}_{\D^m}(z_1,\ldots,z_n))$ and $Aut(\D)\times \mathscr{A}_{\underline{z},I}$
	for any fixed minimal effective tuple $I$ for $\{z_1,\ldots,z_n\}$.
\end{theorem}
\begin{proof}
	We see from Theorem \ref{AutOnCara} that the action of $Aut(\D)$ commutes with the action of any Jordan isomorphism of $\Cnz$. Since the elements of $Aut(\D)$ are uniquely determined by their values at two distinct points and $I$ is minimal, the rest of the proof follows easily from Lemma \ref{JordanIsoAndA_zI}.
\end{proof}
\begin{corollary}
The elements of $Aut(\mathscr{D}_\D (z_1,\ldots,z_n))$ are given by  $(w_1,\ldots,w_n)\mapsto (\varphi(w_{\sigma(1)}),\ldots,\varphi(w_{\sigma(n)}))$ where $\varphi\in Aut(\D)$ and $\sigma$ is permutation for which there is a $\psi\in Aut(\D)$ such that $\psi(z_j)=z_{\sigma(j)}$ for all $j$.
\end{corollary}
\begin{corollary}
	There is a one-to-one correspondence between $Aut(\mathscr{D}_{\D^2}(z_1,\ldots,z_n))$ and one of the following sets
	
	\begin{enumerate}
		\item $Aut(\D)\times \{\varphi\in Aut(\D^2):\varphi\{z_1,\ldots,z_n\}=\{z_1,\ldots,z_n\}\}$, provided that none of the interpolation problems $z_{1j}\mapsto z_{2j}$ and $z_{2j}\mapsto z_{1j}$ are solvable.
		\item $Aut(\D)\times \{\varphi\in Aut(\D):\varphi\{z_{11},\ldots,z_{n1}\}=\{z_{11},\ldots,z_{n1}\}\}$, provided that the interpolation problem $z_{j1}\mapsto z_{j2}$ is solvable.
		\item $Aut(\D)\times \{\varphi\in Aut(\D):\varphi\{z_{12},\ldots,z_{n2}\}=\{z_{12},\ldots,z_{n2}\}\}$, provided that the interpolation problem $z_{j2}\mapsto z_{j1}$ is solvable.
	\end{enumerate}
\end{corollary}

\vspace{0.1in} \noindent\textbf{Acknowledgements:} The work is supported by European Horizon MSCA grant,\\ CZ.02.01.01/00/22 10/0008854.



\begin{thebibliography}{99}
	
	
		\bibitem{Abate89} M. Abate, Marco,\textit{ Iteration theory of holomorphic maps on taut manifolds}, Res. Lecture Notes Math. Complex Anal. Geom., Mediterranean Press, Rende, 1989.
		
		\bibitem{Araz94} J. Arazy, 
		\textit{Isometries of Banach algebras satisfying the von Neumann inequality}, Math. Scand. 74 (1994), no. 1, $137-151$.
		
		\bibitem{ArazSol90} J. Arazy, B. Solel, Baruch, \textit{Isometries of nonselfadjoint operator algebras}, J. Funct. Anal. 90 (1990), no. 2, $284-305$.
		
		\bibitem{BonDun73} F. F. Bonsall, J. Duncan, \textit{Complete normed algebras}, Ergeb. Math. Grenzgeb. Springer-Verlag, New York-Heidelberg, 1973. 
	
	
	\bibitem{Biswas2025} A. Biswas, \textit{Pick interpolation and invariant functions}, Ann. Polon. Math. 134 (2025), no. 3, $247-263$.

\bibitem{CLW1992} B. Cole, K. Lewis, J.  Wermer, \textit{Pick conditions on a uniform algebra and von Neumann inequalities} J. Funct. Anal. 107 (1992), no. 2, $235-254$

\bibitem{CLW1993} B. Cole, K. Lewis, J.  Wermer, \textit{A characterization of Pick bodies}, J. London Math. Soc. (2) 48 (1993), no. 2, $316-328$. 

\bibitem{CW1993} B. Cole, J. Wermer, \textit{Pick interpolation, von Neumann inequalities, and hyperconvex sets}, Complex potential theory (Montreal, PQ, 1993), NATO Adv. Sci. Inst. Ser. C: Math. Phys. Sci., 439, Kluwer Acad. Publ., Dordrecht, 1994, $89-129$.

\bibitem{CW1996} B. Cole, J. Wermer, \textit{Interpolation in the bidisk} J. Funct. Anal. 140 (1996), no. 1, $194-217$.

\bibitem{CW1997} B. Cole, J. Wermer, \textit{Boundaries of interpolation bodies}, Houston J. Math. 23 (1997), no. 3, $499-527$. 

\bibitem{GarciaPalacios} M. Cabrera Garcia, A. Rodriguez Palacios, \textit{Non-associative normed algebras. Vol. 2.}, Encyclopedia Math. Appl., 167, Cambridge University Press, Cambridge, 2018.

\bibitem{HeathSuff81} L. F. Heath, T. J. Suffridge, \textit{Holomorphic retracts in complex n-space}, Illinois J. Math. 25 (1981), no. 1, $125-135$.

\bibitem{KosMcCar20}  L. Kosinski, J. E. McCarthy, \textit{Extensions of bounded holomorphic functions on the tridisk}, Rev. Mat. Iberoam. 36 (2020), no. 3, $791-816$.

\end{thebibliography}
\end{document}